\documentclass{amsart}

\usepackage{amsmath}
\usepackage{amsfonts}
\usepackage{amsopn}
\usepackage{amssymb}
\usepackage{url}

\newcommand{\R}{\ensuremath{\mathbb{R}}}
\newcommand{\K}{\ensuremath{\mathbb{K}}}
\newcommand{\Z}{\ensuremath{\mathbb{Z}}}
\newcommand{\FP}{\ensuremath{\mathcal{F}}}
\newcommand{\BS}{\ensuremath{BS}}
\newcommand{\FV}{\ensuremath{FV}}
\DeclareMathOperator{\supp}{supp}
\DeclareMathOperator{\vol}{vol}
\DeclareMathOperator{\FS}{FSize}
\DeclareMathOperator{\size}{size}

\newtheorem{thm}{Theorem}
\newtheorem{dfn}{Definition}
\newtheorem{lemma}[thm]{Lemma}

\newtheorem{cor}[thm]{Corollary}

\title{Homological and homotopical higher-order filling functions}
\author{Robert Young}
\address{Institut des Hautes \'Etudes Scientifiques,\\
Le Bois Marie, 35 route de Chartres, F-91440 Bures-sur-Yvette, France}
\email{rjyoung@ihes.fr}

\date{\today}

\begin{document}
\begin{abstract}
  We construct groups in which $\FV^3(n)\nsim \delta^2(n)$.  This
  construction also leads to groups $G_k, k\ge 3$ for which
  $\delta^{k}(n)$ is not subrecursive.
\end{abstract}
\bibliographystyle{hamsplain} \maketitle The Dehn function of a group
provides a measure of the complexity of the group's word problem by measuring the
difficulty of filling loops in a corresponding complex.  A natural
generalization is to consider the difficulty of filling
higher-dimensional manifolds or cycles, and there are several ways to
do so, varying in the nature of the filling and the boundary.  One can
consider, for example, the volume necessary to fill a $k$-sphere with
a ball ($\delta^k$), to fill $\partial M$ with $M$ ($\delta^M$), or to
fill a $(k-1)$-cycle by a $k$-chain ($\FV^k$).  In some cases, these
functions are equivalent; for example, the methods used in
\cite{YoungFING} work for all these definitions.  Along these lines,
Brady et al.\ \cite{BradyEtAl} showed that if $\partial M$ is connected and $\dim
M=k+1\ge 4$ then $\delta^M(n)\preceq \delta^k(n)$.  In this note, we
will show that this is not necessarily true if $\dim M=3$, and that
there are groups where $\FV^3$ is not equivalent to $\delta^2$.  We
will also show that for $k\ge 2$ there are groups where $\FV^k$ is not
subrecursive (i.e., $\FV^k$ grows faster than any computable function) and for
$k\ge 3$, there are groups where $\delta^{k}$ is not subrecursive.

We start by defining some filling functions.
To define $\delta^{k}$, we will take the approach of Brady et al.\ 
\cite{BradyEtAl}, which is equivalent to the definition of Alonso,
Wang, and Pride \cite{AlWaPrHi} or of Bridson \cite{BriHigher}.  We
recall their definition of an admissible map:
\begin{dfn}[Admissible maps \cite{BradyEtAl}]
  Let $W$ be a compact $k$-manifold and $X$ a CW-complex.  An {\em
    admissible map} from $W$ to $X$ is a map $f:W\to X^{(k)}$ such
  that $f^{-1}(X^{(k)}\setminus X^{(k-1)})$ is a disjoint union of
  open $k$-dimensional balls in $W$, each mapped homeomorphically to a
  $k$-cell of $X$.  We define the volume $\vol(f)$ of $f$ as the
  number of these balls.
\end{dfn}
If $\alpha=\sum a_i \Delta_i\in C_k(X;\R)$ is a cellular chain in $X$,
with $a_i\in \R$ and $\Delta_i$ distinct cells of $X$, define
$\|\alpha\|_1=\sum|a_i|$.  If $W$ is orientable and $f:W\to X$ is an
admissible map, the image of the fundamental class of $W$ is a
cellular $k$-chain, which we call $\hat{f}$.  This has integer
coefficients, and $\|\hat{f}\|_1\le \vol{f}.$ Furthermore, if $W$ is
closed, then $\hat{f}$ is a cycle.  

If $X$ is a $k$-connected CW-complex, one can define the filling
volume of an admissible map $\alpha:S^k\to X$ as
$$\delta_X^k(\alpha)=\inf \{\vol \beta\mid \beta:D^{k+1}\to X, \beta|_{S^k}=\alpha, \beta \text{ is admissible}\}$$
and the $k$-th order Dehn function of the complex by
$$\delta^k_X(n)=\mathop{\sup_{\alpha:S^k\to X}}_{\vol{\alpha}\le n} \delta_X^k(\alpha),$$
where $\alpha$ is assumed to be admissible.

We can also define the Dehn function of a group:
\begin{dfn}[Dehn functions]\label{DehnFnDef}
  We say a group $G$ is $\FP^k$ if there is a $K(G,1)$ with finite
  $k$-skeleton.  If $G$ is $\FP^{k+1}$, let $X$ be the
  $(k+1)$-skeleton of the universal cover of such a $K(G,1)$ and
  define the {\em $k$-th order Dehn function} of $G$
  $$\delta^k_G(n)=\delta^k_{X}(n).$$
\end{dfn}
This function depends on the choice of $X$, but Gromov's Filling
Theorem \cite{BriInvit} states that the growth rate of $\delta^1_X$ is
an invariant of $G$, and Alonso, Wang, and Pride generalized this to higher dimensions \cite{AlWaPrHi}.  That is, we define the partial ordering
\begin{equation}
  \label{DehnEquivalenceRelation}
  f\preceq g\text{ iff $\exists A,B,C,D,E$ s.t.\ $f(n)\le A 
    g(B n + C)+Dn +E$ for all $n>0$}
\end{equation}
and let $f\sim g$ if and only if $f\preceq g$ and $f\succeq g$.  If
$X_1$ and $X_2$ are as in Definition \ref{DehnFnDef}, then $\delta^k_{X_1}\sim
\delta^k_{X_2}$.

There are several ways to generalize this beyond fillings of spheres
by balls.  Brady et al.~\cite{BradyEtAl} provide one generalization.
If $(M,\partial M)$ is a (smooth or PL) compact manifold pair with
$\dim{M}=k+1$, define the filling volume of an admissible map
$\alpha:\partial M\to X$ as
$$\delta_X^{M}(\alpha)=\inf \{\vol(\beta)\mid \beta:M\to X, \beta|_{\partial M}=\alpha, \beta \text{ is admissible}\}.$$
and
$$\delta_X^{M}(n)=\mathop{\sup_{\alpha:\partial M\to X}}_{\vol{\alpha}\le n} \delta_X^M(\alpha),$$
where $\alpha$ is again assumed to be admissible.  In particular, $\delta_X^{D^{k+1}}=\delta_X^{k}$.

Another generalization is to consider fillings of chains by cycles,
with volume given by $\|\cdot\|_1$-norm.  Gromov \cite{GroAII} defined
the filling volume function $\FV$ of a manifold by using Lipschitz
cycles; we will use cellular cycles.  For $\alpha \in Z_{k-1}(X;\Z)$ a
$(k-1)$-cycle, define
$$\FV_{X,\K}^{k}(\alpha)=\inf \{\|\beta\|_1\mid \beta \in C_{k}(X;\K), \partial \beta=\alpha\}$$
for $\K=\R$ or $\K=\Z$, and define the $k$-dimensional filling volume
function of $X$ by
$$\FV_{X,\K}^{k}(n)=\mathop{\sup_{\alpha\in Z_{k-1}(X;\Z)}}_{\|\alpha\|_1\le n} \FV_{X,\K}^{k}(\alpha).$$

As with $\delta^M$, we can specify the manifold to fill.  If $N$ is a
closed orientable $(k-1)$-dimensional manifold, we define
$$\FV_{X,\K}^N(n)=\mathop{\sup_{\alpha:N\to X}}_{\vol{\alpha}\le n} \FV_{X,\K}^{k}(\hat{\alpha}),$$
where $\alpha$ is assumed to be admissible.
Then if $(M,\partial M)$ is a compact orientable manifold
pair with $\dim(M)=k$,
$$\FV_{X,\Z}^{\partial M}(n)\preceq \delta_X^M(n)$$
$$\FV_{X,\K}^{\partial M}(n)\preceq \FV_{X,\K}^{k}(n).$$

Finally, we define the filling size of a curve.  The size of a chain
is a notion from geometric measure theory which counts the number of
distinct cells in the support of a chain.  The filling size $\FS$
describes the infimal size of a chain filling a curve.  If
$\beta=\sum_{i=1}^r b_i \Delta_i\in C_k(X;\K)$ for $b_i\ne 0\in \K$
and distinct $k$-cells $\Delta_i$ of $X$, let $\size{\beta}=r$.  Let
the support $\supp{\beta}$ of $\beta$ be the minimal subcomplex of $X$
containing the $\Delta_i$.  Define the {\em filling size}
$\FS(\alpha)$ of an admissible loop $\alpha:S^1\to X$ by
$$\FS_X(\alpha)=\min \{\size{\beta} \mid \beta \in C_{2}(X;\R), \partial \beta=\hat{\alpha}\}$$
$$\FS_X(n)=\mathop{\sup_{\alpha:S^1\to X}}_{\vol{\alpha}\le n} \FS_X(\alpha).$$
This represents the number of different $2$-cells of $X$
necessary to support a filling of a loop.  

Like $\delta$, the functions $\FS$ and $\FV$ are defined in terms of a
CW-complex $X$, but can also be defined up to
\eqref{DehnEquivalenceRelation} for a group
\begin{lemma}
  Let $k\ge 1$ and let $X_1$ and $X_2$, be $k$-connected CW-complexes
  such that $G$ acts on $X_i$ cocompactly, properly discontinuously,
  and by automorphisms.  Let $\K=\R$ or $\K=\Z$.  Then $\FS_{X_1}\sim
  \FS_{X_2}$ and $\FV^{k+1}_{X_1,\K}\sim \FV^{k+1}_{X_2,\K}$.
\end{lemma}
\begin{proof}
  It is enough to show that $\FS_{X_1}\precsim \FS_{X_2}$ and
  $\FV^{k+1}_{X_1,\K}\precsim \FV^{k+1}_{X_2,\K}$; the lemma then
  follows by symmetry.

  By the \v{S}varc-Milnor Lemma, $X_1$ and $X_2$ are both
  quasi-isometric to $G$, and thus are quasi-isometric.  By Lemmas 12
  and 13 of \cite{AlWaPrHi}, there are cellular quasi-isometries
  $f:X_1^{(k+1)}\to X_2^{(k+1)}$ and $g:X_2^{(k+1)}\to X_1^{(k+1)}$
  and a cellular homotopy $h:X_1^{(k)}\times [0,1]\to X_1^{(k+1)}$
  such that for all $x\in X_1^{(k)}$, $h(x,0)=(g\circ f)(x)$
  and $h(x,1)=x$.  Furthermore, there is a $c>0$ such that if
  $$f_*:C_*(X_1^{(k+1)};\K)\to C_*(X_2^{(k+1)};\K)$$ is the map induced by $f$,
  then  for all $\sigma\in C_i(X_1^{(k+1)};\K)$, we have
  $$\|f_*(\sigma)\|_1\le c \|\sigma\|_1,$$
  $$\size(f_*(\sigma)) \le c \size(\sigma).$$
  Similar inequalities hold
  when $f$ is replaced by $g$ or $h$.  
  
  If $\alpha\in Z_k(X_1;\K)$, then $f_*(\alpha)$ is a cycle in $X_2$,
  and there is a $(k+1)$-chain $\beta\in C_{k+1}(X_2;\K)$ such that
  $\partial\beta=f_*(\alpha)$ and
  $\|\beta\|_1\le \FV^{k+1}_{X_2,\K}(f_*(\alpha))+1$.  Then
  $$\beta'=g_*(\beta)+h_*(\alpha\times [0,1])$$
  has boundary $\alpha$ and 
  \begin{align*}
    \FV^{k+1}_{X_1,\K}(\alpha)&\le \|\beta'\|_1\\
    &\le c \FV^{k+1}_{X_2,\K}(f_*(\alpha))+c+ c\|\alpha\|_1 \\
    &\le c \FV^{k+1}_{X_2,\K}(c\|\alpha\|_1)+c+ c\|\alpha\|_1 
  \end{align*}
  Thus $\FV^{k+1}_{X_1,\K}\precsim \FV^{k+1}_{X_2,\K}$.  Similarly, choosing
  $\beta$ so that $\size(\beta)=\FS_{X_2}(f_*(\alpha))$ shows that
  $\FS_{X_1,\K}\precsim \FS_{X_2,\K}$.
\end{proof}
If $G$ acts on $X$ in this way, we can define $\FS_G=\FS_X$ and
$\FV^*_{G,\K}=\FV^*_{X,\K}$, and these functions are well-defined up
to the equivalence relation \eqref{DehnEquivalenceRelation}.  For all
of these functions, we will omit the group when there is no confusion.

We will use $\FS$ to provide a lower bound on some
higher-dimensional filling volumes.
\begin{thm} \label{thm:mainthm}
  If a group $G$ is $\FP^{k+1}$, then
  $$\FV^{(S^1)^{k}}_{G^k,\Z}(n)\succeq \FS_G(n^{1/k})$$
\end{thm}

\begin{proof}
  Let $X$ be the $(k+1)$-skeleton of the universal cover of a $K(G,1)$
  with finite $(k+1)$-skeleton, so that $X$ is $k$-connected and $G$
  acts cocompactly, properly discontinuously, and by automorphisms on
  $X$.  Let $\gamma:S^1\to X$ be an admissible map such that
  $\vol{\gamma}=n$ and $\FS_X(\gamma)=\FS_X(n)$.  Define the map
  $\alpha=\gamma^{k}:(S^1)^{k} \to X^{k}$, where $X^k$ is given the
  product CW-structure; then $\vol(\alpha)=n^{k}$.  Let $\beta\in
  C_{k+1}(X^{k};\Z)$ be a chain whose boundary is $\hat{\alpha}$ and
  such that
  $$\|\beta\|_1=\FV^{(S^1)^{k}}_{X^k,\Z}(\alpha).$$
  Let $p_1,\dots,p_{k}:X^{k}\to X$ be the maps projecting to each
  factor.  Then $p_i(\supp(\beta))$ is a subcomplex of $X$ for each
  $i$.  We claim that for some $i$, this subcomplex supports a 2-chain
  filling $\hat{\gamma}$ and thus has at least $\FS_X(\gamma)$ cells.

  We proceed by contradiction, assuming that $\hat{\gamma}$ is not a
  boundary (over $\R$) in any of the $p_i(\supp{\beta})$'s.  In this
  case, $\hat{\gamma}$ represents a non-zero element of
  $H_1(p_i(\supp{\beta});\R)$ and by the universal coefficient
  theorem, there is a cohomology class $v_i\in
  H^1(p_i(\supp{\beta});\R)$ such that $v_i(\hat{\gamma})= 1$.  Let
  $w_i=p_i^*(v_i)\in H^1(\supp{\beta};\R)$.  We claim that
  $$\biggl[\bigcup_{i=1}^k w_i\biggr](\hat{\alpha})\ne 0;$$
  this contradicts the fact that $\partial \beta=\hat{\alpha}$.
  
  The $w_i$ pull back under $\alpha$ to the standard generators of
  $H^1((S^1)^k;\R)$, so their cup product is a class in $H^{k}(\supp
  \beta;\R)$ which pulls back to a generator $t$ of $H^k((S^1)^k;\R)$.
  Therefore,
  $$\biggl[\bigcup_{i=1}^k w_i\biggr](\hat{\alpha})$$
  is equal to $t$
  evaluated on the fundamental class of $(S^1)^k$, and is thus
  non-zero.  Since $\hat\alpha$ is a boundary in $\supp \beta$, this
  is impossible; any class in $H^{k}(\supp \beta;\R)$ must evaluate to
  0 on $\hat{\alpha}$.  Thus $\gamma$ is a boundary in
  $p_i(\supp{\beta})$ for some $i$.

  This implies that $p_i(\supp{\beta})$ contains at least
  $\FS_X(\gamma)$ 2-cells.  Each of these is the image of a cell of
  $\supp{\beta}$, so $\supp{\beta}$ contains at least $\FS_X(\gamma)$
  cells.  By the definition of a CW-complex, any $(k+1)$-cell of
  $X^{k}$ is contained in a finite subcomplex of $X^{k}$.  Since there
  are only finitely many equivalence classes of cells of $X^{k}$ under
  the action of $G^k$, there is a constant $c>0$ such that for all
  $\sigma\in C^*(X^{k};\R)$, the number of cells in $\supp{\sigma}$ is
  at most $c\size{\sigma}$.  Thus
  $$\FV^{(S^1)^{k}}_{X^k,\Z}(\alpha)= \|\beta\|_1\ge \size{\beta}\ge c^{-1}\FS_X(\gamma)$$
  so
  $$\FV^{(S^1)^{k}}_{X^k,\Z}(n^k)\ge c^{-1}\FS_X(n)$$
  as desired.
\end{proof}

\begin{lemma}\label{lem:nonRecFS}
  There is an aspherical group $G$ for which $\FS(n)$ is not
  subrecursive.
\end{lemma}
\begin{proof}
  Collins and Miller \cite{ColMil} constructed a group $G$ with
  unsolvable word problem and an aspherical presentation.  This group
  is constructed from a free group by applying three successive
  HNN-extensions where the associated subgroups are finitely generated
  free groups.  This group belongs to the hierarchy $\mathcal{F}$
  constructed by Gersten \cite{GerstenHomology}, and Thm.~4.3 of
  \cite{GerstenHomology} states that for groups in this hierarchy,
  $\delta_G$ is bounded by a recursive function of $\FV_{G,\R}^{S^1}$.
  Since $G$ has unsolvable word problem, $\delta_G$ is not
  subrecursive, so $\FV_{G,\R}^{S^1}$ is also not subrecursive.

  We claim that $\FV_{G,\R}^{S^1}(n)$ is bounded by a recursive
  function of $\FS_G(n)$ and thus that $\FS_G(n)$ is not subrecursive.
  Let $X$ be the CW-complex corresponding to a finite aspherical
  presentation of $G$, let $c$ be the total length of the relators in
  the presentation and let $\alpha$ be a loop in $X$.  Let $\beta\in
  C_2(\supp{\beta};\R)$ be a 2-chain such that $\partial\beta=\alpha$
  and such that $\size(\beta)=\FS_X(\alpha)$.  It suffices to show
  that there is a 2-chain $\gamma$ such that $\partial\gamma=\alpha$
  and $\|\gamma\|_1$ is bounded by a recursive function of
  $\size{\beta}$.

  Note that
  \begin{align*}
    \dim C_2(\supp{\beta};\R)& = \size{\beta}\\
    \dim C_1(\supp{\beta};\R)&\le c \size{\beta}
  \end{align*}
  and
  $\partial:C_2(\supp{\beta};\R)\to C_1(\supp{\beta};\R)$ is a linear
  map.  The 2-cells of $\supp{\beta}$ correspond to a basis of
  $C_2(\supp{\beta};\R)$ and the 1-cells correspond to a basis of
  $C_1(\supp{\beta};\R)$.  In these bases, the equation
  $\partial\gamma=\alpha$ corresponds to a system of at most
  $c\size{\beta}$ linear equations in $\size{\beta}$ variables with
  integer coefficients between $-c$ and $c$.  Since $\beta$ is a
  solution, the system is solvable, and since such a system can be
  solved algorithmically, there is a solution whose $\|\cdot\|_1$ norm
  is bounded by a recursive function of $\size{\beta}$.  Thus
  $\FV_{X,\R}^{S^1}(n)$ is bounded by a recursive function of $\FS_X(n)$
  and so $\FS_G(n)$ is not subrecursive.
\end{proof}

Using Theorem \ref{thm:mainthm} and the following theorem of Brady et
al.\ \cite[Rem. 2.6.(4)]{BradyEtAl}, we can give lower bounds on filling functions of products of $G$:
\begin{thm}\label{BradyEtAlThm}
  If $\dim M=k+1\ge 4$, then $\delta^M\le \delta^{k}$ provided
  $\partial M$ is connected or $\delta^{k}$ is superadditive.
\end{thm}

\begin{cor}\label{cor:nonrecursive}
  For the group $G$ in Lemma \ref{lem:nonRecFS}, $\FV^{k+1}_{G^{k},\Z}(n)$ is not
  subrecursive for $k\ge 1$ and $\delta^{k}_{G^{k}}(n)$ is not
  subrecursive for $k\ge 3$.
\end{cor}

Note that the spheres with large filling volumes may be extremely
distorted.  If $\alpha:(S^1)^{k} \to X^{k}$ is as in the proof of
Theorem \ref{thm:mainthm}, then the construction in the proof of
Theorem \ref{BradyEtAlThm} results in a map $\alpha'$ whose image
contains non-recursively large 2-discs filling curves in the image of
$\alpha$.  These discs do not add to the $k$-volume of
$\alpha'$, but they increase its ``complexity''; for instance, the
Lipschitz constants of $\alpha$ and the number of simplices in a
simplicial approximation grow non-recursively with $\vol{\alpha}$.  Brady et
al.'s theorem suggests that considering fillings of spheres is not
particularly restrictive in high dimensions, since low-volume,
high-complexity spheres can be used to approximate arbitrary
manifolds.  To study differences in filling different manifolds, it
may be worthwhile to study other filling functions.

Finally, a theorem of Papasoglu \cite{Papasoglu} states that
$\delta^2_G$ is a subrecursive function for any group $G$ which is
$\FP^3$.  Combining this with Corollary \ref{cor:nonrecursive}, we
obtain the following:
\begin{cor}
  There is a group $G$ such that $\FV^3_{G,\Z}(n)\nsim \delta^2_G(n)$.
\end{cor}

There are also examples of such groups with solvable word problem.  To
construct one such example, we let $G=\BS(1,2)$ and consider $G\times
G$.  One method to show that $\delta_{G}(n)\succeq 2^n$ uses the
asphericity of a certain 2-dimensional $K(G,1)$ to show that discs
filling certain curves must contain an exponentially large number of
2-cells \cite[7.4]{ECHLPT}.  This method also shows that
$\FS_{G}(n)\succeq 2^n$ and thus, by Theorem \ref{thm:mainthm},
$$\FV^{3}_{G\times G,\Z}(n)\succeq 2^{\sqrt{n}}.$$
On the other hand, since there is a 2-dimensional $K(G,1)$, 
$$\delta^{2}_{G}(n)\preceq n,$$
and by Theorem 5.3 of \cite{AlBoBuPrWa2nd},
$$\delta^{2}_{G\times G}(n)\preceq n^2.$$

In these groups, filling a torus takes substantially more volume than
filling a sphere largely because the fundamental group of the torus is
nontrivial.  It would be interesting to see if there are other ways
that the topology of the boundary affects the difficulty of filling.
In particular, it remains open to find examples of groups in which
filling a genus $g$ surface is harder than filling a torus.

The author would like to thank Hanna Bennett, Max Forester, and the
referee for their comments on drafts of this paper, and NYU for its
hospitality during part of the preparation of this paper.

\providecommand{\bysame}{\leavevmode\hbox to3em{\hrulefill}\thinspace}
\providecommand{\href}[2]{#2}

\end{document}